\documentclass[12pt]{article}

\usepackage[english]{babel}
\usepackage[cp1251]{inputenc}
\usepackage[T2A]{fontenc}

\usepackage{amssymb}
\usepackage{amsmath}
\usepackage{amsthm} 
\usepackage{mathtext}
\usepackage{amsfonts}
\usepackage{graphicx,geometry}
\usepackage[title]{appendix}
\usepackage{cite}

\renewcommand{\leq}{\leqslant}
\renewcommand{\geq}{\geqslant}

\DeclareSymbolFont{myletters}{OML}{ztmcm}{m}{it}
\DeclareMathSymbol{\uplambda}{\mathord}{myletters}{"15}

\newtheorem{theorem}{Theorem}
\newtheorem{lemma}{Lemma}

\newtheorem{corollary}{Corollary}

\theoremstyle{definition}
\newtheorem{definition}{Definition}
\newtheorem{remark}{Remark}

\newenvironment{enbibliography}{\vspace{-0.5cm}\begin{thebibliography}}{\end{thebibliography}}

\begin{document} 
\title{On the sufficient conditions for the S-shaped Buckley--Leverett function}
\author{N.~V.~Rastegaev \\ 
\small{St. Petersburg State University} \\ \small{
7/9 Universitetskaya nab., St. Petersburg, 199034 Russia} 
\\ \small{rastmusician@gmail.com}}
\renewcommand{\today}{}
\maketitle
\abstract{
The flux function in the Buckley--Leverett equation, that is, the function characterizing the ratio of the relative mobility functions of the two phases, is considered. The common conjecture stating that any convex mobilities result in an S-shaped Buckley--Leverett function is analyzed and disproved by a counterexample. Additionally, sufficient conditions for the S-shaped Buckley--Leverett function are given. The class of functions satisfying those conditions is proven to be closed under multiplication. Some functions from known relative mobility models are confirmed to be in that class.
}



\section{Introduction}
In fluid dynamics, the Buckley--Leverett equation is one of the simplest conservation laws used to model two-phase flow in porous media. The equation is given by:
\[
s_t + f(s)_x = 0,
\]
where $s = s(x,t)$ is the water saturation and $f$ is the fractional flow function, also known as the Buckley--Leverett function. This function characterizes the ratio of relative mobilities of the two phases, which is expressed as:
\[
f(s) = \dfrac{m_a(s)}{m_a(s) + m_b(1-s)},
\]
where $m_a$ and $m_b$ represent the relative phase mobilities, typically water and oil mobilities in the context of petroleum sciences. These mobility functions are often increasing and convex.

The assumption that the Buckley--Leverett function is S-shaped is prevalent. This is likely due to the fact that an S-shaped function is used as the only example in the principal work by Buckley and Leverett \cite{BL}, and this assumption is repeated in many papers thereafter. This assumption is no longer critical in the case of the Buckley--Leverett equation, since the Riemann problem for it can be solved analytically for any function $f$ by the convex hull construction given by Oleinik (see \cite{Oleinik, Gelfand}). Nonetheless, it is still often important in more general conservation systems that include more phases or components, or have additional parameters such as temperature (see \cite{JnW,Bahetal,Castetal,Castaetal2,Castaetal3,Wahetal}).
However, there is no comprehensive research on when $f$ is actually S-shaped. The prevalent conjecture among engineers is that convex mobilities result in an S-shaped fractional flow function. Some mathematicians hold similar expectations. The only paper known to the author (and the one that inspired this work) investigating sufficient conditions for the S-shaped function is the paper by Casta\~{n}eda \cite{Castaneda}. That paper proves that when relative phase mobilities are convex power functions, the resultant Buckley--Leverett function is S-shaped. It also states that the author could not find a counterexample to the convex conjecture. Also noteworthy is \cite[Claim 17]{Castetal}. It does not give any conditions for when $f$ is S-shaped, but instead allows us to circumvent the question altogether and find the solution to the Riemann problem regardless, but only when the initial discontinuity connects pure or almost pure states ($s=0$ and $s=1$).

In this paper, we present a counterexample where convex relative phase mobilities produce a fractional flow function that has more than one inflection point. Furthermore, we provide sufficient conditions for the S-shaped Buckley--Leverett function. The paper has the following structure: Section \ref{sec:suff} defines the class of convex mobilities and proves the theorem asserting that mobilities from that class always give an S-shaped fractional flow function. Additionally, it proves the proposed class of convex functions is closed under multiplication. Section \ref{sec:counter} presents the counterexample of two functions outside the previously defined class that produce a fractional flow function that is not S-shaped. Appendix \ref{ap:A} contains the graphs illustrating the counterexamples provided in this paper. Appendix \ref{ap:B} applies the proposed conditions to some known relative mobility models.

\section{Sufficient conditions for the S-shaped function}
\label{sec:suff}
\begin{definition}\label{main_def}
Let $\mathcal{M}$ be a set of functions $m\in \mathcal{C}^2[0,1]$ such that
\begin{itemize}
    \item[(C1)] $m(s)>0$ for $s>0$, $m(0)=0$;
    \item[(C2)] $m'(s)>0$ for $s>0$, $m'(0)=0$;
    \item[(C3)] $m''(s)>0$ for $s>0$;
    \item[(C4)] $\dfrac{m''}{m'}$ is a decreasing function on $(0,1)$.
\end{itemize}
\end{definition}

\begin{remark}\label{rem:Corey}
Any power function $m(s) = As^a$ with $A>0$ and power $a > 1$ is in $\mathcal{M}$, since $\dfrac{m''}{m'} = \dfrac{a-1}{s}$ is decreasing. Thus Theorem \ref{main_theorem} below covers the result of \cite[Theorem 409]{Castaneda}.
\end{remark}

\begin{lemma}\label{lemma_on_C4}
Let $m\in\mathcal{M}$. Then $\dfrac{m'}{m}$ is also a decreasing function. Therefore, for all $m\in\mathcal{M}$ the following variation of (C4) holds:
\begin{itemize}
    \item[\textnormal{(C4*)}] \textnormal{ $\dfrac{m'}{m}$ and $\dfrac{m''}{m'}$ are decreasing functions on $(0,1)$.}
\end{itemize}
\end{lemma}
\begin{proof}
To prove $\dfrac{m'}{m}$ is a decreasing function we need to demonstrate
\[
\left( \dfrac{m'}{m} \right)' = \dfrac{m''m - m'^2}{m^2} < 0,
\]
which, considering (C1)--(C3), is equivalent to
\[
\dfrac{m'}{m} > \dfrac{m''}{m'},
\]
which, considering (C4), follows from Cauchy's Mean Value Theorem:
\[
\dfrac{m'(x)}{m(x)} = \dfrac{m'(x) - m'(0)}{m(x) - m(0)} = \dfrac{m''(\widetilde{x})}{m'(\widetilde{x})} > \dfrac{m''(x)}{m'(x)}
\]
holds for all $x\in(0,1)$ and certain $\widetilde{x} \in (0, x)$.
\end{proof}

\begin{remark}
Though we require only $\mathcal{C}^2$ smoothness from our functions, in the proof below we will operate the third derivative as if it is regular, but only in the context of examining the local monotonicity of the second derivative. Any such argument could be easily modified to avoid the mention of the third derivative, but that would make the proof needlessly cumbersome. 
\end{remark}

\begin{theorem}\label{main_theorem}
Let $m_a$ and $m_b$ be two mobility functions from the class $\mathcal{M}$. Then the fractional flow function $f(s) = \dfrac{m_a(s)}{m_a(s)+m_b(1-s)}$ is S-shaped, that is, there exists a unique inflection point $s^*\in(0,1)$, such that $f''(s^*) = 0$.
\end{theorem}
\begin{proof}
We note that $(m_b(1-s))' = -m'_b(1-s)$, thus $m_b$ changes sign with every derivative. Keeping that in mind, we omit the variables $s$ and $1-s$ in the notation hereafter, implying that $m_b$ and its derivatives are applied to the variable $1-s$.

Denote $m = m_a + m_b$. We solve the equation
\begin{equation}\label{main_eq}
f'' = \dfrac{(m''_am_b-m_am''_b)m - 2m'(m'_am_b+m_am'_b)}{m^3} = 0,
\end{equation}
and aim to demonstrate that it has a unique solution on $(0,1)$. 
Note immediately that if $m''_am_b-m_am''_b$ and $-m'$ have the same sign, then $f''$ has it too. Thus, any solution of \eqref{main_eq} must satisfy
\begin{equation}\label{sign_restriction}
(m''_am_b-m_am''_b)m' > 0 \quad \text{or} \quad 
m''_am_b-m_am''_b = m' = 0.
\end{equation}
Note also that $m'$ is increasing due to (C3), since $m'' = m''_a + m''_b > 0$, and it changes sign exactly one time from negative to positive. Denote by $s_1$ the sign change point:
\[
m' < 0, \quad s<s_1; \qquad m'>0, \quad s>s_1.
\]
Similarly, due to (C4*) we deduce that $\dfrac{m_a''}{m_a} = \dfrac{m_a''}{m'_a} \dfrac{m_a'}{m_a}$ is decreasing and, keeping in mind the difference in the variable, $\dfrac{m_b''}{m_b} = \dfrac{m_b''}{m'_b} \dfrac{m_b'}{m_b}$ is increasing. Thus, $\dfrac{m_a''}{m_a} - \dfrac{m''_b}{m_b}$ is a decreasing function and has a single sign change from positive to negative. Denote the sign change point $s_2$ and note that
\[
m''_am_b - m_am''_b > 0, \quad s<s_2; \qquad m''_am_b - m_am''_b < 0, \quad s>s_2.
\]
Due to the restriction \eqref{sign_restriction}, we know that any solution of \eqref{main_eq} must be between $s_1$ and $s_2$. If $s_1 = s_2$, then $s^* = s_1 = s_2$ is the unique solution of \eqref{main_eq}, and the theorem is proved. Otherwise, we consider separately the cases $s_1 < s_2$ and $s_1 > s_2$.

\underline{Case $s_1 < s_2$.} On $(s_1, s_2)$ we have 
\[
m' > 0, \quad 
m''_am_b - m_am''_b > 0. 
\]
It is clear that $f''(s_1) > 0$ and $f''(s_2) < 0$, thus the solution exists. In order to prove its uniqueness, we
denote 
\[
h = m_a'm_b + m_a m'_b,
\]
keeping in mind that in the current case $h'=m''_am_b - m_am''_b>0$. Using this new notation, we rewrite
\begin{equation}\label{ddfrewrite}
f' = \dfrac{h}{m^2}, \qquad f'' = \dfrac{h'm - 2m'h}{m^3}.
\end{equation}
Note that
\[
\dfrac{h'}{m_a'm_b'} = \dfrac{m_a''m_b - m_a m''_b}{m_a'm_b'} = \dfrac{m_a''}{m_a'}\dfrac{m_b}{m_b'} - \dfrac{m_b''}{m_b'}\dfrac{m_a}{m_a'} 
\]
is decreasing due to (C4*), thus
\[
m_a'm_b' \left( \dfrac{h'}{m_a'm_b'} \right)' = 
h'' - h' \left[ \dfrac{m_a''}{m_a'} - \dfrac{m_b''}{m_b'} \right] < 0.
\]
Note also that
\[
\dfrac{m''}{m'} + \dfrac{m'}{m} > \dfrac{m''}{m'} = \dfrac{m''_a + m''_b}{m'_a - m'_b} > \dfrac{m_a''}{m_a'} > \dfrac{m_a''}{m_a'} - \dfrac{m_b''}{m_b'},
\]
thus combining these relations we obtain
\begin{equation}\label{h_est}
h'' - h' \left[ \dfrac{m''}{m'} + \dfrac{m'}{m} \right] < 0.
\end{equation}
Let $s=s^*$ be a solution of \eqref{main_eq}, that is, $f''(s^*) = 0$. Then from \eqref{ddfrewrite} we have $h'(s^*)m(s^*) = 2m'(s^*)h(s^*)$. Therefore, 
\[
f'''(s^*) = \left.\dfrac{(h'm^2 - 2mm'h)'}{m^4}\right|_{s=s^*} -  \dfrac{4m'(s^*)}{m(s^*)} f''(s^*).
\]
The second term is zero, and the first could be estimated using \eqref{h_est}:
\begin{align*}
\left.\dfrac{(h'm^2 - 2mm'h)'}{m^2}\right|_{s=s^*} & = {}
h''(s^*) - 2\dfrac{m'(s^*)}{m(s^*)} h(s^*) \left[\dfrac{m''(s^*)}{m'(s^*)} + \dfrac{m'(s^*)}{m(s^*)} \right] \\
& {} = h''(s^*) - h'(s^*) \left[\dfrac{m''(s^*)}{m'(s^*)} + \dfrac{m'(s^*)}{m(s^*)} \right] < 0.
\end{align*}
Therefore, $f'''(s^*) < 0$ and $f''$ changes sign from positive to negative at $s^*$. There could only be one solution with that property, so $s^*$ is unique.

\underline{Case $s_1 > s_2$.} In this case on $(s_2, s_1)$ we have 
\[
m' < 0, \quad 
h' = m''_am_b - m_am''_b < 0. 
\]
Now $f''(s_2) > 0$ and $f''(s_1) < 0$, thus the solution exists.
The steps of the uniqueness proof are the same.
We note that 
\[
\dfrac{m''}{m'} + \dfrac{m'}{m} < \dfrac{m''}{m'} = \dfrac{m''_a + m''_b}{m'_a - m'_b} < -\dfrac{m_b''}{m_b'} < \dfrac{m_a''}{m_a'} - \dfrac{m_b''}{m_b'},
\]
therefore \eqref{h_est} still holds. Other than that, no modifications are required, thus the theorem is proved.

\end{proof}

In practice, the following theorem is very helpful in verifying condition (C4) for some common functions.

\begin{theorem}\label{thm_closed}
Let $m_1$ and $m_2$ be positive and increasing functions on $(0,1)$ satisfying (C4*). Then their product $m_1m_2$ also satisfies (C4*).
\end{theorem}
\begin{proof}
In this proof, we will use the notation $i = 1, 2$. Note that since $\dfrac{m_i'}{m_i}$ and $\dfrac{m_i''}{m_i'}$ are decreasing, their product $\dfrac{m_i''}{m_i}$ is also decreasing. Consider the derivatives of these fractions and we obtain
\[
m_i''m_i - (m_i')^2 < 0, \quad m_i'''m_i' - (m_i'')^2 < 0, \quad m_i'''m_i - m_i''m_i' < 0.
\]
It is easy to see that 
\[
\left(\dfrac{(m_1m_2)'}{m_1m_2}\right)' = 
(\ln (m_1 m_2))'' =
(\ln m_1)'' + (\ln m_2)'' = \left(\dfrac{m_1'}{m_1}\right)' + \left(\dfrac{m_2'}{m_2}\right)' < 0.
\]
Sadly, the same simple trick does not work for the second fraction. Instead, we prove that $\dfrac{(m_1m_2)''}{(m_1m_2)'}$ is decreasing by expanding the derivatives and grouping some terms to achieve a similar estimate:
\begin{align*}
(m_1m_2)'''&(m_1m_2)' - ((m_1m_2)'')^2 \\
& {} = (m_1'''m_2 + 3m_1''m_2' + 3m_1'm_2'' +m_1m_2''')(m_1'm_2+m_1m_2') \\
& {} - (m_1''m_2+2m_1'm_2'+m_1m_2'')^2 \\
& {} = m_2^2 (m_1'''m_1' - (m_1'')^2) + m_1^2 (m_2'''m_2' - (m_2'')^2) \\
& {} + m_2'm_2 (m_1'''m_1 - m_1''m_1') + m_1'm_1 (m_2'''m_2 - m_2''m_2') \\
& {} + (m_2')^2 (m_1''m_1 - (m_1')^2) + (m_1')^2 (m_2''m_2 - (m_2')^2) \\
& {} - 2(m_1''m_1 - (m_1')^2)(m_2''m_2 - (m_2')^2) < 0.
\end{align*}
\end{proof}

\begin{corollary}
$\mathcal{M}$ is closed under multiplication.
\end{corollary}

\section{Counterexample for the convex conjecture}
\label{sec:counter}
The first counterexample we constructed goes as follows. Consider the phase mobilities
\[
m_a(s) = m_b(s) = s^{1.1}e^{s^{10}}.
\]
It is easy to calculate the derivatives:
\[
m_a'(s) = 1.1s^{0.1}e^{s^{10}} + 10s^{10.1}e^{s^{10}} > 0, 
\]
\[
m_a''(s) = 0.11s^{-0.9}e^{s^{10}} + 112s^{9.1}e^{s^{10}} + 
100s^{19.1}e^{s^{10}} > 0,
\]
and it is clear that these mobilities admit (C1), (C2) and (C3).
However, function 
\[
\dfrac{m_a''}{m_a} = 0.11 s^{-2} + 112 s^{8} + 100 s^{18}
\]
changes monotonicity, since
\[
\left(\dfrac{m_a''}{m_a}\right)' = s^{-3} (1800 s^{20} + 896 s^{10} -0.22) = 0
\]
has a solution $s \approx 0.4355$, thus (C4) is not upheld. That leads to $f$ having $3$ inflection points: $s=0.5$ is an inflection point due to symmetry, but it has the wrong monotonicity, that is, $f'''(0.5) > 0$. Thus there must exist two additional inflection points, one before $0.5$ and one after. Though it is clear on the graph of $f''$ that it has three zeroes, the graph of $f$ is less obvious, barely deviating from the diagonal (see Fig.\ref{Fig1} and Fig.\ref{Fig2} in Appendix \ref{ap:A}). But the first counterexample allowed us to find a family of similar functions with more pronounced graphs. One of them is
\[
m_a(s) = m_b(s) = s^{1.1}(1 + 15 s^{10}).
\]
See Fig.\ref{Fig3} and Fig.\ref{Fig4} in Appendix \ref{ap:A} for the corresponding plots. The principles upon which these examples were constructed are generalized in the following theorem.

\begin{theorem}
Let $m_a = m_b$ be mobility function satisfying (C1)--(C3) from Definition \ref{main_def}. Additionally, let
\begin{equation}\label{counterexample_property}
\left.\left( \dfrac{m_a''}{m_a^3} \right)'\right|_{s=0.5} > 0.
\end{equation}
Then the corresponding fractional flow function has more than one inflection point and therefore is not S-shaped.
\end{theorem}
\begin{proof}
Note that due to symmetry, we have
\[
m'(0.5) = h'(0.5) = 0.
\]
Therefore
\[
f''(0.5) = \left. \dfrac{h'm - 2m'h}{m^3} \right|_{s=0.5} = 0,
\]
thus, $s=0.5$ is an inflection point. However,
\begin{align*}
f'''(0.5) & = \left. \left(\dfrac{h'm - 2m'h}{m^3}\right)' \right|_{s=0.5} =
\left. \dfrac{h''m - 2m''h}{m^3} \right|_{s=0.5} \\
&{} =  \left. \dfrac{m_a'''m_a^2 - 3 m_a'' m_a' m_a}{2m_a^3} \right|_{s=0.5} = \dfrac{m_a^2}{2} \left.\left( \dfrac{m_a''}{m_a^3} \right)'\right|_{s=0.5} > 0,
\end{align*}
so $f''$ changes sign from negative to positive at $s=0.5$, therefore it cannot be a unique inflection point and there must exist at least two more inflection points, one before $0.5$ and one after.
\end{proof}
It is a fairly straightforward task to check that \eqref{counterexample_property} holds for both counterexamples provided above. It is also clear that \eqref{counterexample_property} contradicts (C4). But the class of functions satisfying neither (C4) nor \eqref{counterexample_property} is vast, and it leads to questions on the rigidity of (C4). Is it possible to construct a wider class $\mathcal{M}$, for which Theorem \ref{main_theorem} holds, by weakening the (C4) restriction? We leave it an open problem for now. What is clear is that \eqref{counterexample_property} is not the only way to construct a counterexample, just the most direct one. It is possible to construct mobilities that do not satisfy \eqref{counterexample_property} but still lead to additional inflection points. To give an example, functions $m_a(s) = m_b(s) = s^{1.1}(1+15s^{30})$ result in $f'''(0.5) < 0$, but $f$ still has $5$ inflection points, as clearly shown on Fig.~\ref{Fig5} and Fig.~\ref{Fig6}.

\begin{appendices}

\section{}
\label{ap:A}

The graphs for the counterexample $m_a(s) = m_b(s) = s^{1.1}e^{s^{10}}$:
\begin{figure}[!htbp]
  \centering
  \begin{minipage}[b]{0.49\textwidth}
    \includegraphics[width=\textwidth]{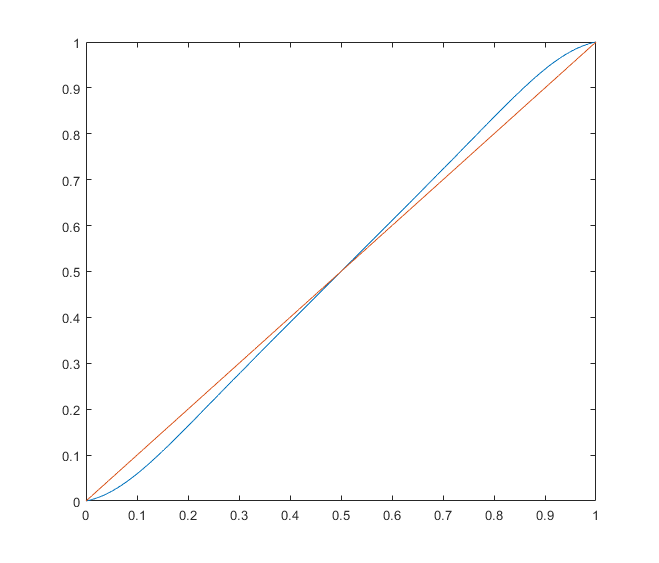}
    \caption{Function $f$.}
    \label{Fig1}
  \end{minipage}
  \begin{minipage}[b]{0.49\textwidth}
    \includegraphics[width=\textwidth]{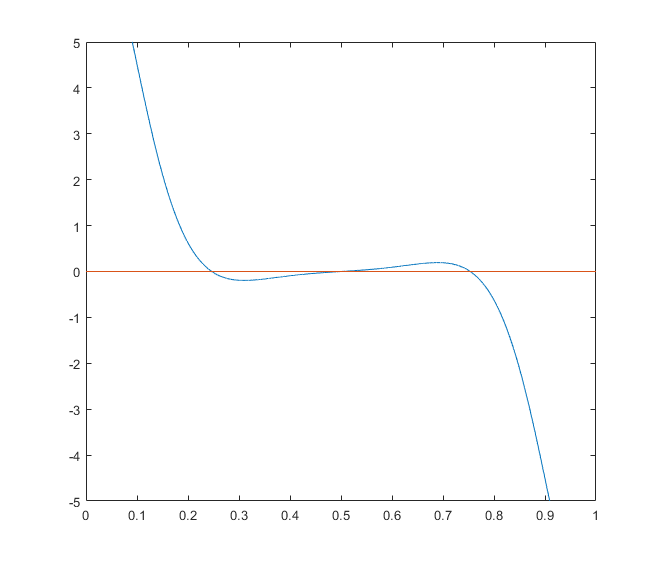}
    \caption{$f''$ with $3$ zeroes.}
    \label{Fig2}
  \end{minipage}
\end{figure}

\newpage

The graphs for the counterexample $m_a(s) = m_b(s) = s^{1.1}(1+15s^{10})$:
\begin{figure}[!htbp]
  \centering
  \begin{minipage}[b]{0.49\textwidth}
    \includegraphics[width=\textwidth]{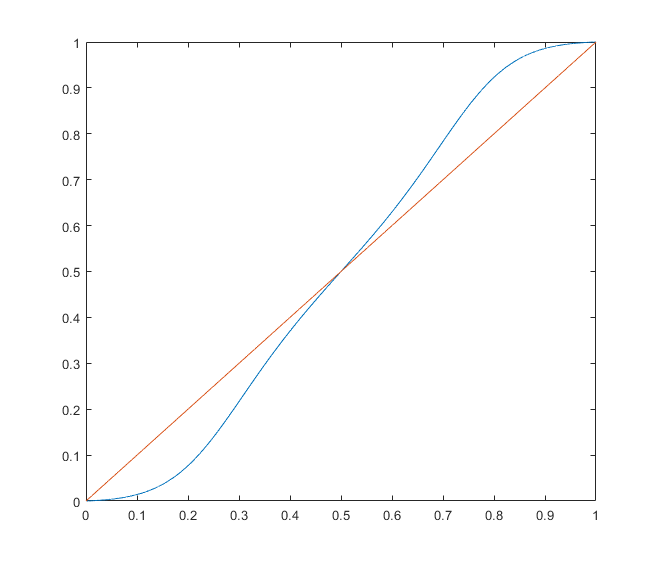}
    \caption{Function $f$.}
    \label{Fig3}
  \end{minipage}
  \begin{minipage}[b]{0.49\textwidth}
    \includegraphics[width=\textwidth]{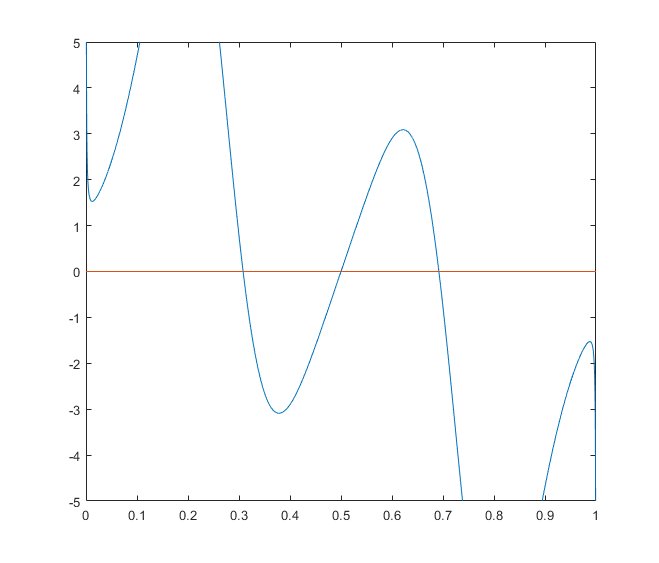}
    \caption{$f''$ with $3$ zeroes.}
    \label{Fig4}
  \end{minipage}
\end{figure}

The graphs for the counterexample $m_a(s) = m_b(s) = s^{1.1}(1+15s^{30})$:
\begin{figure}[!htbp]
  \centering
  \begin{minipage}[b]{0.49\textwidth}
    \includegraphics[width=\textwidth]{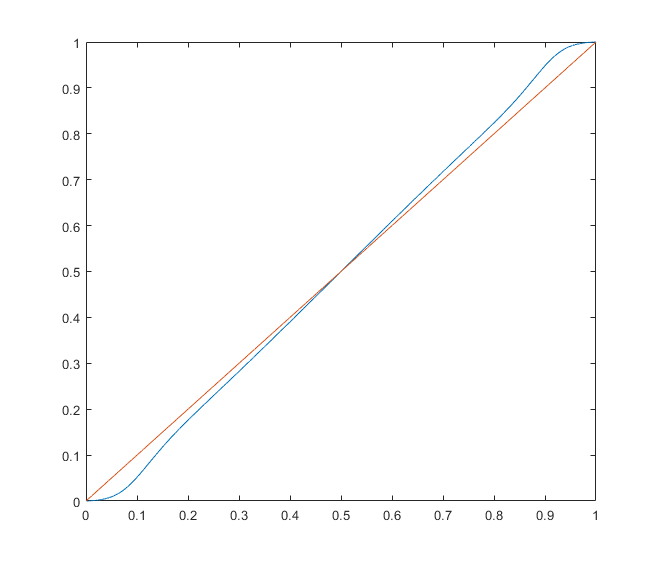}
    \caption{Function $f$.}
    \label{Fig5}
  \end{minipage}
  \begin{minipage}[b]{0.49\textwidth}
    \includegraphics[width=\textwidth]{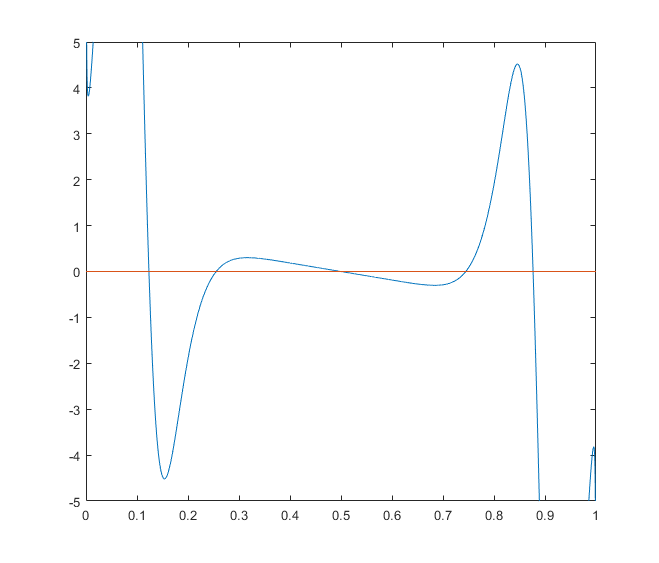}
    \caption{$f''$ with $5$ zeroes.}
    \label{Fig6}
  \end{minipage}
\end{figure}

\newpage

\section{}
\label{ap:B}
In this appendix, we check various known relative mobility models against the conditions (C1)--(C4).

\subsection{Modified Brooks and Corey model}
The simplest and the most commonly used relative mobility model (see \cite{Lake}) is the power law model
\[
m_a(s) = As^a, \quad m_b(s) = Bs^b,
\]
where $A, B > 0$, $a, b > 1$. As we already noted in Remark \ref{rem:Corey}, any power function with power greater than $1$ satisfies (C1)--(C4), therefore $m_a, m_b \in \mathcal{M}$.

\subsection{Brooks and Corey model}
The original model proposed by Corey (see \cite{Corey}) is
\[
m_a(s) = s^4, \quad m_b(s) = s^2(1 - (1-s)^2).
\]
Later, Brooks and Corey introduce a parameter to generalize that model (see \cite{BrooksCorey}). That resulted in the following functions:
\[
m_a(s) = s^{\frac{2+3\lambda}{\lambda}}, \quad m_b(s) = s^2 \left(1 - (1-s)^{\frac{2+\lambda}{\lambda}}\right),
\]
where $\lambda$ is the pore size distribution index. This was further generalized in \cite{Chen} to include an additional parameter:
\[
m_a(s) = s^{\eta + \frac{2+\lambda}{\lambda}}, \quad m_b(s) = s^\eta \left(1 - (1-s)^{\frac{2+\lambda}{\lambda}}\right).
\]
In all variations, $m_a$ is a power function, and thus already considered previously. To study $m_b$ we first analyze the expression in the brackets.

\begin{lemma}
Function $n_\alpha(s) = 1 - (1-s)^\alpha$ on $(0,1)$ is positive, increasing and satisfies (C4*) for all $\alpha>1$.
\end{lemma}
\begin{proof}
\[
n'_\alpha(s) = \alpha(1-s)^{\alpha-1} > 0, \quad n''_\alpha(s) = -\alpha(\alpha-1)(1-s)^{\alpha-2} <0. 
\]
Therefore,
\[
n''_\alpha n_\alpha - (n'_\alpha)^2 < 0,
\]
so $n'_\alpha/n_\alpha$ is decreasing and
\[
\dfrac{n''_\alpha}{n_\alpha} = \dfrac{1 - \alpha}{1-s}
\]
is also a decreasing function.
\end{proof}
Therefore, due to Theorem \ref{thm_closed} we obtain (C1), (C2) and (C4) for $m_b$ for all $\eta > 1$ and $\lambda > 0$. The only condition left to check is (C3). Denote $\alpha = \frac{2+\lambda}{\lambda}$.
\[
m''_b(s) = \eta(\eta-1)s^{\eta-2} \left(1 - (1-s)^{\alpha}\right) + 2\eta\alpha s^{\eta-1} (1-s)^{\alpha-1} -\alpha(\alpha-1)s^\eta(1-s)^{\alpha-2}.
\]
It is clear that for $1 < \alpha < 2$ or $\lambda > 2$ the last term blows up near $s=1$, thus (C3) is broken. Otherwise, we can prove $m_b$ is convex.
\begin{lemma}
Let $\alpha, \eta \geq 2$. Then $m_b(s) = s^\eta \left(1 - (1-s)^\alpha\right)$ is convex on $(0,1)$.
\end{lemma}
\begin{proof}
We rewrite
\begin{equation}\label{ddmb}
m''_b(s) = s^{\eta-2}\Big[\eta(\eta-1) - P(s)(1-s)^{\alpha-2}\Big],
\end{equation}
where
\[
P(s) = \eta(\eta-1) - 2\eta\gamma s + \gamma(\gamma+1) s^2, \qquad \gamma = \eta + \alpha - 1.
\]
It is clear that $m''_b$ is positive near $0$ and $1$. It is also easy to see that
\[
(\eta(\eta-1) - P(s)(1-s)^{\alpha-2})' = (1-s)^{\alpha-3} ((\alpha-2)P(s) - (1-s)P'(s)),
\]
and
\begin{equation}\label{two_zeros}
(\alpha-2)P(s) - (1-s)P'(s) = \alpha \left(\gamma (\gamma+1) s^2 -  2\gamma(\eta+1) s + \eta(\eta+1)\right) = 0
\end{equation}
has at most two zeroes on $(0,1)$. Therefore, the expression in the brackets in \eqref{ddmb} has at most two extremums on $(0,1)$, and if we show them to be positive, the proof will be concluded. Let $z\in(0,1)$ be a root of \eqref{two_zeros}. Note that \eqref{two_zeros} gives us
\[
P(z) = 2\gamma z - 2\eta.
\]
Using this and Bernoulli's inequality, we estimate
\begin{align*}
\eta(\eta-1)(1-z)^{2-\alpha} - P(z)  & \geq \eta(\eta-1) + \eta(\eta-1)(\alpha-2)z - 2\gamma z + 2\eta \\
& {} = (\eta - 2z)(\eta+1) + (\alpha-2)(\eta-2)(\eta+1)z > 0.
\end{align*}

\end{proof}
Therefore, $m_b \in\mathcal{M}$ for all $\alpha, \eta \geq 2$ (or equivalently  $0<\lambda \leq 2, \eta \geq 2$).

\subsection{Chierici model}
Chierici (see \cite{Chierici}) proposed an exponential law for the relative mobility functions:
\[
m_a(s) = A \exp \left[-B\left( \dfrac{s}{1-s} \right)^{-M} \right].
\]
This expression is often not convex, so we are going to provide an example of Chierici functions satisfying our conditions. Let $M = 1$, $B > 2$. Then
\[
m_a(s) = A \exp \left[-B \dfrac{1-s}{s}  \right], \quad
m_a'(s) = \dfrac{AB}{s^2} \exp \left[-B \dfrac{1-s}{s}  \right] > 0,
\]
\[
m_a''(s) = \dfrac{AB^2 - 2ABs}{s^4} \exp \left[-B \dfrac{1-s}{s}  \right] > 0,
\]
\[
\left(\dfrac{m'_a}{m_a}\right)' = -\dfrac{2B}{s^3} < 0, \quad 
\left(\dfrac{m''_a}{m'_a}\right)' =
\left(\dfrac{B-2s}{s^2}\right)' = \dfrac{2s-2B}{s^3} < 0.
\]
Therefore $m_a \in \mathcal{M}$ for $M=1$ and all $B>2$.

\end{appendices}

\section*{Acknowledgements}

The author thanks F.~Bakharev, A.~Enin, A.~Nazarov and Yu.~Petrova for the fruitful discussions of the problem.

The work is supported by the Ministry of Science and Higher Education of the Russian
Federation (agreement no. 075-15-2022-287).
\bigskip
\bigskip

\begin{enbibliography}{99}
\addcontentsline{toc}{section}{References}

\bibitem{BL} Buckley, S.~E. and Leverett, M., 1942. Mechanism of fluid displacement in sands. Transactions of the AIME, 146(01), pp.~107-116.

\bibitem{Oleinik}
Oleinik, O.~A., 1957. Discontinuous solutions of non-linear differential equations. Uspekhi Matematicheskikh Nauk, 12(3)(75), pp.~3-73 (in Russian). English translation in American Mathematical Society Translations, 26(2), 1963, pp.~95-172.

\bibitem{Gelfand}
Gelfand, I.~M., 1959. Some problems in the theory of quasilinear equations. Uspekhi Matematicheskikh Nauk, 14(2), pp.~87-158 (in Russian). English translation in Transactions of the American Mathematical Society, 29(2), 1963, pp.~295-381.

\bibitem{JnW}
Johansen, T. and Winther, R., 1988. The solution of the Riemann problem for a hyperbolic system of conservation laws modeling polymer flooding. SIAM journal on mathematical analysis, 19(3), pp.~541-566.

\bibitem{Bahetal}
Bakharev, F., Enin, A., Petrova, Y. and Rastegaev, N., 2021. Impact of dissipation ratio on vanishing viscosity solutions of the Riemann problem for chemical flooding model. arXiv preprint arXiv:2111.15001.

\bibitem{Castetal}
Casta\~{n}eda, P., Furtado, F. and Marchesin, D., 2013. The convex permeability three-phase flow in reservoirs. IMPA Preprint S\'{e}rie E-2258, pp.~1-34.

\bibitem{Castaetal2}
Casta\~{n}eda, P., Abreu, E., Furtado, F. and Marchesin, D., 2016. On a universal structure for immiscible three-phase flow in virgin reservoirs. Computational Geosciences, 20, pp.~171-185.

\bibitem{Castaetal3}
Tang, J., Casta\~{n}eda, P., Marchesin, D. and Rossen, W.~R., 2019. Three-Phase Fractional-Flow Theory of Foam-Oil Displacement in Porous Media With Multiple Steady States. Water Resources Research, 55(12), pp.~10319-10339.

\bibitem{Wahetal}
Wahanik, H., Eftekhari, A.~A., Bruining, J., Marchesin, D. and Wolf, K.~H., 2010, October. Analytical solutions for mixed CO2-water injection in geothermal reservoirs. In Canadian Unconventional Resources and International Petroleum Conference. OnePetro.

\bibitem{Castaneda}
Casta\~{n}eda, P., 2016. Dogma: S-shaped. The Mathematical Intelligencer, 38, pp.~10-13.

\bibitem{Lake}
Lake, L.W., 1989. Enhanced oil recovery, Chapter 3.

\bibitem{Corey}
Corey, A.~T., 1954. The interrelation between gas and oil relative permeabilities. Producers Monthly 19 (November), pp.~38-41.

\bibitem{BrooksCorey}
Brooks, R.~H. and Corey, A.~T., 1964. Hydraulic Properties of Porous Media. Hydrology Papers, No. 3, Colorado State U., Fort Collins, Colorado.

\bibitem{Chen}
Chen, D., Pan, Z., Liu, J. and Connell, L.~D., 2013. An improved relative permeability model for coal reservoirs. International Journal of Coal Geology, 109, pp.~45-57.

\bibitem{Chierici}
Chierici, G.~L., 1984. Novel relations for drainage and imbibition relative permeabilities. Society of Petroleum Engineers Journal, 24(03), pp.~275-276.

\end{enbibliography}

\end{document}